\documentclass[11pt]{article}

\usepackage[active]{srcltx}
\usepackage{amsmath}
\usepackage{amssymb}
\usepackage{amscd}
\usepackage{amsthm}
\usepackage[latin1]{inputenc}
\usepackage{mathrsfs}

\usepackage{eufrak}

\newtheorem{teo}{Theorem}
\newtheorem{prop}[teo]{Proposition}

\theoremstyle{definition}
\newtheorem{defin}[teo]{Definition}

\newcounter{poppa}
\newtheorem{say}[poppa]{}

\newcommand{\ra}{\rightarrow}
\newcommand{\C}{\mathbb{C}}
\newcommand{\R}{\mathbb{R}}
\newcommand{\Zeta}{{\mathbb{Z}}}

\newcommand{\meno}{^{-1}}

\newcommand{\alfa}{\alpha}

\newcommand{\la}{\lambda}

\newcommand{\restr}[1]          {\vert_{#1}}

\newcommand{\cinf}{C^\infty}
\newcommand{\om}{\omega}

\renewcommand{\phi}{\varphi}
\newcommand{\lds}{\ldots}

\newcommand{\sx}{\langle}
\newcommand{\xs}{\rangle}

\newcommand{\lra}{\longrightarrow}

\newcommand{\OO}{\mathcal{O}}

\newcommand{\Keler}             {K\"{a}hler }

\newcommand{\debar }            {\bar {\partial } }

\newcommand{\chern}             {\operatorname{\mathrm{c}}}

\newcommand{\PP}{\mathbb{P}}

\newcommand{\Vol}{\operatorname{Vol}}
\newcommand{\Ric}{\operatorname {Ric} }

\newcommand{\su}{\mathfrak{su}}
\newcommand{\SU}{\operatorname{SU}}
\newcommand{\SO}{\operatorname{SO}}

\newcommand{\dual}[1][\cdot]{|#1|_{h^*}}
\newcommand{\hn}{h^{\otimes N}}
\newcommand{\vm}{dV_M}
\newcommand{\vx}{dV_X}
\newcommand{\xsl}{\rangle_{L^2}}
\newcommand{\xslx}{\rangle_{L^2(X)}}
\newcommand{\hats}{\hat{s}}
\newcommand{\hate}{\hat{e}}
\newcommand{\suv}{\su(V_N)}
\newcommand{\tr}{\operatorname{tr}}
\newcommand{\Mom}{\Phi}
\newcommand{\hardy}{\mathcal{H}^2}
\newcommand{\harm}{\mathcal{H}}
\newcommand{\hardyx}{\mathcal{H}^2(X)}
\newcommand{\scal}{\langle \, , \rangle}

\newcommand{\pai}{ \left (}
\newcommand{\ring}{ \right )}
\newcommand{\ptn}{Q_N}
\newcommand{\pzn}{P^0_N}
\newcommand{\id}{\operatorname{id}}
\newcommand{\omfsn} {\om_{N}}
\newcommand{\cinfr}{C^\infty(M,\R)}
\newcommand{\susn}{\mathscr{A}_N}
\newcommand{\susno}{\susn^0}
\newcommand{\dn}{{d_N}}
\newcommand{\T}{t}
\newcommand{\tp}{\hat{P}}

\begin{document}

\author{Alessandro Ghigi}

\title{On the approximation of functions\\
on a Hodge manifold}
\maketitle

\begin{abstract}
  If $(M, \om)$ is a Hodge manifold and $f\in \cinf(M)$ we construct a
  canonical sequence of real algebraic functions $f_N$ such that $f_N
  \to f$ in the $\cinf$ topology.  The definition of $f_N$ is inspired
  by Berezin-Toeplitz quantization and by ideas of Donaldson. The
  proof follows quickly from known results of Fine, Liu and Ma.
\end{abstract}

\section{Introduction}

Let $M$ be a complex projective algebraic manifold and let $\om$ be a
Hodge metric, i.e. a \Keler metric such that $[\om /2\pi]$ lies in the
image of $H^2(M, \Zeta) $ inside $H^2(M,\R)$.  Using ideas of
Donaldson \cite[\S 4]{donaldson-numerical-results}, Fine, Liu and Ma
\cite{fine-Calabi} have already introduced a canonical way to
approximate smooth functions on $M$. More precisely, using the Fourier
modes of the Szeg\"o kernel they defined a sequence of operators
\begin{gather*}
  \ptn : \cinf(M) \ra \cinf(M) \qquad N=1, 2, 3 , \lds
\end{gather*}
with the property that for any $f\in \cinf(M)$ one has $Q_Nf \to f$ in
the topology of $\cinf(M)$, see \eqref{eq:QN} below and Theorem
\ref{teo-liu-ma}.

The purpose of this note is to introduce another sequence of
operators, denoted $P_N$, which are a variant of the $\ptn$.  These
operators still have the property that $P_N f \to f$ in $\cinf(M)$,
but have also two other nice features: they admit a simple geometric
interpretation in terms of linear systems and moment map, and the
range of each $P_N$ is a finite dimensional subspace of $\cinf(M)$
whose elements are real algebraic functions. 

\smallskip

The author wishes to thank Roberto Paoletti for various crucial
discussions on the Szeg\"o kernel.  He also thanks Alberto Della
Vedova, Christine Laurent and Xiaonan Ma for useful emails.  Finally
he acknowledges financial support from the PRIN 2007 MIUR ``Moduli,
strutture geometriche e loro applicazioni''.

\section{The geometric construction}

Let $W$ be a finite-dimensional complex vector space. Any $w\in W$ can
be identified with a section $s_w\in H^0\left(\PP(W^*), \OO_{\PP(W^*)}
  (1)\right) $ defined by $s_w([\la]) (t \la) = t \la (w)$. Here
$t\la$ is the generic element of $\OO_{\PP(W^*), [\la]} (-1)$.  If
$\scal$ is a Hermitian product on $W$, the line bundle $
\OO_{\PP(W^*)} (1)$ inherits a Hermitian metric $h_{FS}$ such that
\begin{gather*}
\bigl  | s_w([\la])\bigr|_{h_{FS}} = \frac{|\la(w)|}{|\la|}.
\end{gather*}
If $\Theta (h_{FS})$ denotes the curvature of $h_{FS}$, then $\om_{FS}
= i \Theta(h_{FS})$ is by definition the Fubini-Study metric on
$\PP(W^*)$ induced by the product $\scal$.  The group $\SU(W,\scal)$
acts on $(\PP(W^*), \om_{FS})$ both holomorphically and isometrically
and
\begin{gather*}
  \Mom : \PP(W^*) \ra \su(W) \qquad \Mom([\la]) = i \biggl ( \frac{
    \sx \cdot , w \xs w }{|w|^2} - \frac{\id_W}{\dim W} \biggr )
\end{gather*}
is the moment map.  (Here $w$ is the vector such that $\la (\cdot) =
\sx \cdot , w\xs $.) This means that it is equivariant and that for
any $A\in \su(W)$
$$ 
d \pai \Mom, A\ring = -i_{\xi_Av}\om_{FS},
$$
where $\pai X, Y \ring := -\tr XY$ is the Killing product on $\su(W)$
and $\xi_A$ is the fundamental vector field corresponding to $A$.
Notice that
\begin{gather}
  \label{eq:momento-Pn}
  \pai \Mom\left( [\la] \right) , A \ring = -i \frac{\sx Aw, w\xs
  }{|w|^2} \qquad A\in \su(W),\ \la = \sx \cdot, w\xs \in W^*.
\end{gather}

Let $M^m$ be a projective manifold with an ample line bundle $L\ra M$
and let $h$ be a Hermitian metric on $L$. Denote by $\Theta(h)$ the
curvature of the Chern connection of $(L,h)$ and assume that
$\om:=i\Theta(h) \in 2\pi\chern_1(L)$ is a \Keler\ form.  Let $N$ be a
natural number.  Using $h$ and the volume form $\vm :=\om^m/m!$ one
can endow the space $V_N:=H^0(M, L^N)$ with the $L^2$-Hermitian
product
\begin{gather*}
  \sx s_1 , s_2\xsl : = \int_M \hn (s_1 (z), s_2(z) ) \vm(z) \quad s_1,
  s_2 \in V_N.
\end{gather*}
Correspondingly, the projective space $\PP(V^*_N)$ is endowed with the
Fubini-Study metric induced from $\sx\, , \xsl$.  Denote by
\begin{gather*}
  \phi_N : M \ra \PP(V^*_N).
\end{gather*}
the map associated to the complete linear system of $L^{ N}$, by
$\omfsn$ the Fubini-Study metric on $\PP(V_N^*)$ induced by Hermitian
product $\scal_{L^2}$ and by $\Mom_N : \PP(V_N^*) \ra \su(V_N)$ the
moment map of $(\PP(V_N^*), \omfsn)$ as described above.

Let $L^2(M, L^N)$ be the Hilbert space of $L^2$ sections of $L^N$.
For any $N$ let $\Pi_N : L^2(M, L^N) \ra V_N$ be the orthogonal
projector onto $V_N=H^0(M, L^N)$.  A function $f\in \cinf(M)$ induces
a sequence of \emph{Toeplitz} operators
\begin{gather*}
  T_{f, N} : L^2(M, L^N) \ra L^2(M,L^N) \\ T_{f,N} s := \Pi_N (f\cdot
  \Pi_Ns) \qquad s\in L^2(M, L^N) .
\end{gather*}
Since the range of $T_{f,N}$ is $V_N$ and $T_{f,N}\equiv 0 $ on
$V_N^\perp$, we will identify $T_{f,N}$ with its restriction to $V_N$
considered as an endomorphism of $V_N$. If $s\in V_N$, then $ T_{f,N}
s := \Pi_N (fs)$.  If $f$ is \emph{real valued}, then $T_{f,N}$ is
self-adjoint.  
Set 
$$
\dn= h^0(M, L^N)-1
$$
and let $T_{f,N}^0$ be the trace-free part of $T_{f,N}$:
\begin{gather*}
  T_{f,N}^0 := T_{f,N} - \frac{ \tr T_{f,N}}{\dn+1} \id_{V_N}.
\end{gather*}
Then $iT^0_{f,N} \in \su(V_N)$ and $(\,\cdot\, , i T^0_{f,N} )$ is a
linear function on $\su(V_N)$. The pull-back of this function to $M$
via the map $\Mom_N \circ \phi_N : M \ra \suv$ is denoted by $\pzn f$.
More explicitly
\begin{gather}
  \label{eq:pzn-1}
  \pzn f\, (z) := \pai \Mom_N \left (\phi_N (z)\right ), i T^0_{f, N}
  \ring.
\end{gather}
So we have defined a sequence of operators $\pzn : \cinfr \ra
\cinfr$. These operators are linear since the map $f \mapsto T_{f,N}$
is linear.  

\section{Relation with the Szeg\"o kernel}

The operators $\pzn$ that we just defined (and their kins $P_N$ to be
defined below) admit a simple description in terms of the Fourier
modes of the Szeg\"o kernel. This will enable to deduce very quickly
some of their analytic properties from known results on the Szeg\"o
kernel.

We start by recalling some important facts regarding the Szeg\"o
kernel.  Set $\Vol(M) = \int_M \vm = (2\pi)^m \chern_1(L)^m/m!$.  Let
$h^*$ denote the Hermitian metric on $L^*$ induced from $h$ and let
$\dual$ denote the corresponding norm.  Define $ \rho : L^* \ra \R $
by $\rho(\la): = \dual[\la]^2 -1$ and set
\begin{gather*}
  D:=\{\la \in L^* : \rho(\la) <0\}=\{\la \in L^* : \dual[\la] <1\}\\
  X:=\{\la \in L^* : \rho(\la) =0\}=\{\la \in L^* : \dual[\la] =1\} .
\end{gather*}
$D$ is a strictly pseudoconvex domain in $L^*$ since $\om = i \Theta $
is positive.  $X=\partial D$ is a smooth hypersurface in $L^*$ and is
a principal $S^1$-bundle over $M$.  Set $\alfa := i \debar \rho
\restr{X}$ and
\begin{gather*}
  \vx:= \frac{\alfa \wedge (d \alfa)^m}{2\pi m!} = \frac{\alfa \wedge
    \pi^*\vm}{2\pi}.
\end{gather*}
If $\rho_N$ denotes the representation of $S^1$ on $\C$ given by
multiplication by $e^{iN\theta}$, then the associated bundle
$X\times_{\rho_{-N}} \C$ is isomorphic to $ L^N$ via the maps
\begin{gather}
  \label{correspondence}
  % X\times_{\rho_{-N}} \C \ra L \qquad
  [\la,z] \mapsto z \cdot \left ( \la\meno(1) \right )^{\otimes N}
  \qquad u \mapsto [\la, \la^{\otimes N} (u) ].
\end{gather}
Therefore a section $s$ of $L^N$ corresponds to an equivariant
function $\hat{s}: X \ra \C$ defined by $\hat{s} (x) = x^{\otimes N}
(s(\pi(x)))$. The equivariance means that
\begin{gather}
  \label{eq:equivariance}
  \hat{s}( e^{i\theta} \la) = \rho_{-N}(e^{-i\theta}) \hat{s}(\la) =
  e^{iN\theta} \hat{s}(\la).
\end{gather}
Denote by $\hardyx$ the space of $L^2$ functions on $X$ that are
annihilated by the Cauchy-Riemann operator $\debar_b$ (see
\cite[p. 321]{zelditch-theorem-Tian}).  Since $D$ is strictly
pseudoconvex $\hardyx$ coincides with the space of $L^2$--boundary
values of holomorphic functions on $D$. It is a closed $S^1$-invariant
subspace of $L^2(X)$, hence splits as a direct sum $\hardyx =
\oplus_N\hardy_{N\in \Zeta}(X)$, where $\hardy_N(X)$ is the set of
functions $\hat{s}$ in $\hardyx$ that satisfy \eqref{eq:equivariance}.
Via the correspondence \eqref{correspondence} the holomorphic sections
of $L^N$ correspond to elements of $\hardy_N(X)$. If $s_1, s_2 \in
\cinf(M, L^N)$, then
\begin{gather*}
  \sx s_1, s_2 \xsl = \int_X \hat{s}_1 \overline{\hat{s}_2 } \vx= \sx
  \hats_1, \hats_2 \xslx.
\end{gather*}
If $x\in X$ the linear functional
\begin{gather*}
  V_N \lra \C \qquad s \mapsto x^{\otimes N}\left (s\left (
      \pi(x)\right) \right) = \hats(x)
\end{gather*}
is represented by a section $e_{x,N} \in V_N$. Equivalently, $\hats(x)
= \sx s, e_{x,N} \xsl = \sx \hats, \hate_{x,N} \xslx$ for any $s\in
V_N$.  The sections $e_{x,N}$ are called \emph{coherent states} in the
mathematical physics literature.

We use $\Pi_N$ also to denote the projection $\Pi_N : L^2(X) \ra
\hardy_N(X)$. Let $\Pi_N(x,x')$ be the Schwartz kernel of $\Pi_N$ (we
use $\vx$ to identify functions with densities).  In the following
Proposition we gather various well-known and elementary properties of
$\Pi_N$.
\begin{prop}
  (a) $\Pi_N \in \cinf(X\times X)$. (b) $\Pi_N(x,x') =
  \overline{\Pi_N(x', x)}$.  (c) For any $x\in X$, $\hate_{x,N} =
  \Pi_N( \cdot, x)$.  (d) There are functions $K_N \in \cinf(M\times
  M)$ and $E_N \in \cinf(M)$ such that if $x, x'\in X$, $z=\pi(x),
  z'=\pi(x') $, then
  \begin{gather*}
    K_N(z,z') = |\Pi_N(x, x')|^2 \qquad E_N(z) = \Pi_N(x,x).
  \end{gather*}
  (e) $K_N(z,z') = K_N(z',z)$.  (f) If $\pi(x) = z$ and
  $\{s_j\}_{j=0}^\dn$ is an orthonormal basis of $V_N$, then
  \begin{equation}
    \label{eq:EN}
    E_N(z) = \sum_{j=0}^\dn |s_j(z)|_h^2 = |e_{x,N}|^2_{L^2}.
  \end{equation}
  (g) For $N$ large enough the function $E_N $ is strictly positive.
\end{prop}
\begin{proof}
  Let $\{s_j\}_{j=0}^\dn$ be an orthonormal basis of $V_N$. Then
  \begin{gather}
    \label{proietto}
    \Pi_N(x, x') = \sum_{j=0}^\dn \hats_j(x) \overline{\hats_j(x')}
  \end{gather}
  This proves (a) and (b).  If $e_{x,N} = \sum_j a_j s_j$, then
  \begin{gather*}
    \overline{a_j} = \sx s_j, e_{x,N} \xsl =
    % x^{\otimes N} \left ( s_j (\pi(x) \right ) =
    \hats_j(x).
  \end{gather*}
  Hence $e_{x,N} = \sum_j \overline {\hats_j(x)} s_j$. Together with
  \eqref{proietto} this proves (c).  Observe that
  \begin{gather*}
    \Pi_N(e^{i\theta} x , e^{i\theta'} ) = e^{iN(\theta - \theta')}
    \Pi_N(x, x').
  \end{gather*}
  Hence the function $|\Pi_N|^2$ is $S^1\times S^1$--invariant on
  $X\times X$ and descends to a smooth function $K_N \in \cinf(M\times
  M)$.  Similarly $\Pi_N(x,x)$ is $S^1$-invariant on $X$ and descends
  to a smooth function $E_N\in \cinf(M)$. This proves (d). (e) follows
  from (b).  To prove (f) it is enough to observe that if $x\in X$,
  then $|x^{\otimes N}|_{h^*} =1$, hence $|s_j(z)|_h = |x^{\otimes N }
  \left ( s_j(z) \right ) | = |\hats_j (x)|$.  This yields immediately
  the first equality. For the same reason
  \begin{gather*}
    |e_{x,N}|^2_{L^2} = \sum_{j=0}^{\dn} \bigl | \sx e_{x,N} , s_j
    \xsl \bigr |^2 = \sum_{j=0}^{\dn} \left | x^{\otimes N} \left (
        s_j (z) \right )\right |^2 = \sum_{j=0}^{\dn} | s_j (z) |^2
    _h.
  \end{gather*}
  The positivity of $E_N$ is equivalent to the fact that $L^N$ has no
  base points, so (g) follows from the ampleness of $L$.
\end{proof}
\begin{teo}
  If $f\in \cinfr$ and $z\in M$, then
  \begin{gather}
    \label{eq:pzn-2}
    \pzn f (z) = \int_M \frac{K_N(z,z')}{E_N(z) } f(z') \vm(z') -
    \frac{ \tr T_{f,N}}{\dn + 1}
    \\
    \label{eq:traccia}
    \tr T_{f,N} = \int_M E_N(z) f(z) \vm (z).
  \end{gather}
\end{teo}
\begin{proof}
  Let $x\in X $ be such that $z=\pi(x)$. Set $\la (\cdot) = \sx \cdot,
  e_{x,N} \xsl$. Then $\la \in V_N^*$ and $\phi_N(z) = [\la]$. By
  \eqref{eq:pzn-1} and \eqref{eq:momento-Pn}
  \begin{gather*}
    \pzn f (z) = \pai \Mom_N \left ( \phi_N(z ) \right ) , i T^0_{f,N}
    \ring = -i \frac{\sx i T_{f,N}^0 e_{x,N}, e_{x,N}\xsl
    }{|e_{x,N}|_{L^2}^2}=\\
    = \frac{\sx T_{f,N} e_{x,N}, e_{x,N}\xsl }{|e_{x,N}|_{L^2}^2} -
    \frac{\tr T_{f,N}} {\dn + 1}\\
    \sx T_{f,N} e_{x,N} , e_{x,N} \xsl = \int_M f(z') \, | e_{x,N}
    (z')|^2_h\ \vm(z').
  \end{gather*}
  If $z'=\pi(x')$, then $|e_{x,N}(z') |_h^2 = |\hate _{x,N} (x')|^2 =
  |\Pi_N(x', x)|^2 = K_N(z',z) = K_N(z,z')$.  Hence
  \begin{gather*}
    \sx T_{f,N} e_{x,N} , e_{x,N} \xsl = \int_M K_N(z,z')
    f(z')\vm(z').
  \end{gather*}
  By \eqref{eq:EN} $|e_{x,N}|^2_{L^2}= E_N(z)$, so
  \begin{gather*}
    \frac{\sx T_{f,N} e_{x,N}, e_{x,N}\xsl }{|e_{x,N}|_{L^2}^2} =
    \frac{1}{E_N(z) } \int_M K_N(z,z') f(z') \vm(z')=\\
    = \int_M \frac{K_N(z,z')}{E_N(z) } f(z') \vm(z').
  \end{gather*}
  This proves \eqref{eq:pzn-2}.  To prove \eqref{eq:traccia} observe
  that if $\{s_j\}_{j=0}^\dn$ is an orthonormal basis of $V_N$, then
  \begin{gather*}
    \tr T_{f,N} = \sum_{j=0}^\dn \sx T_{f,N} s_j, s_j \xsl =
    \sum_{j=0} ^\dn \sx f  s_j, s_j \xsl =\\
    = \sum_{j=0}^\dn \int_M f(z') |s_j(z')|^2_h \vm(z') = \int_M f(z')
    E_N(z') \vm(z').
  \end{gather*}
  In the last line we have used \eqref{eq:EN}.
\end{proof}
\begin{defin}
  For $N=1, 2, 3, \lds$ let $P_N : \cinfr \ra \cinfr$ be the linear
  operator defined by
  \begin{gather*}
    P_N f (z) := \int_M \frac{K_N(z,z')}{E_N(z) } f(z') \vm(z') .
  \end{gather*}
\end{defin}

In \cite[\S 4]{donaldson-numerical-results} Donaldson introduced
operators $Q_N : \cinf(M) \ra \cinf(M)$, defined by
\begin{gather}
  \label{eq:QN}
  Q_Nf(z) = \frac{\Vol(M)}{\dn+1} \int_M K_N(z,z') f(z')\vm(z').
\end{gather}
(His definition is actually more general since integration is
performed with respect to a measure that does not necessarily coincide
with $\vm$.)  In \cite[Appendix, Thm. 26]{fine-Calabi} Kefeng Liu and
Xiaonan Ma proved the following result.
\begin{teo}
  [Liu-Ma]
  \label{teo-liu-ma}
  For every $k$ there is a constant $C_k$ such that
  \begin{gather}
    \label{eq:liu-ma}
    ||Q_N f - f ||_{C^k(M)} \leq \frac{C_k}{N} ||f||_{C^{k+1}(M)}.
  \end{gather}
\end{teo}
In \cite[p. 519] {fine-Calabi} the estimate is given in terms of
$||f||_{C^k(M)} $ on the right hand side, due to a small inaccuracy in
the computations. The correct formulation, as above, is given in
\cite[p. 1 note 1] {ma-marinescu-Berezin-Toeplitz}. I thank Professor
Ma for pointing this out to me.

It is important to notice that the operators $Q_N$ represent the
derivative of the map $\om \mapsto \om_N$, see
\cite[pp. 495-496]{fine-Calabi}.

\begin{teo} \label{main} (a) There is $C$ such that
  \begin{gather*}
    \left|\frac{ \tr T_{f,N}}{\dn +1} - \frac{1}{\Vol(M)} \int_M f(z)
      \vm(z) \right | \leq \frac{C}{N} ||f||_\infty.
  \end{gather*}
  (b) For any $k$ there is $C_k$ such that
  \begin{gather*}
    ||P_N f - f ||_{C^k(X)} \leq \frac{C_k}{N} ||f||_{C^{k+1}(X)}.
  \end{gather*}
\end{teo}
\begin{proof}
  The first statement follows from a fundamental result of Zelditch
  \cite{zelditch-theorem-Tian}: the sequence $\{E_N\}_N$ admits an
  asymptotic development in $\cinf(M)$ of the form
  \begin{gather}
    \label{eq:zelditch}
    E_N = \left (\frac{N}{2\pi}\right )^m + O(N^{m-1}).
  \end{gather}
  (Apply \cite[Cor. 2]{zelditch-theorem-Tian} with $G=\Ric(h)$.)  By
  Riemann-Roch
  \begin{equation}
    \label{eq:zelditch-2}
    \begin{gathered}
      \dn + 1 = h^0(M,L^N) =\frac{ \chern_1(L)^m }{m!}  \, N^m + O
      (N^{m-1}) =\\
      = \frac{\Vol(M)}{(2\pi)^m} \, N^m + O (N^{m-1}).
    \end{gathered}
  \end{equation}
  Using \eqref{eq:traccia} we get
  \begin{gather*}
    \frac{E_N}{\dn + 1} = \frac{1}{\Vol(M)} + O (\frac{1}{N})
  \end{gather*}
  in $\cinf(M)$. Integrating against $f$ yields (a).  To prove (b) it
  is enough to observe that
  \begin{gather*}
    P_N f = \frac{\dn + 1 }{\Vol(M) \cdot E_N} Q_Nf.
  \end{gather*}
  Therefore \eqref{eq:liu-ma} and \eqref {eq:zelditch-2} yield the
  result.
\end{proof}

\section{Final remarks}

% \noindent{}\textbf{Remarks.}

\begin{say}
  Tian approximation theorem \cite{tian-polarized,bouche-convergence,
    zelditch-theorem-Tian, catlin-Tian} asserts that
  $\phi_N^*\omfsn\big /N \to \om$ in the $\cinf$ topology.  The
  metrics $\phi_N^* \omfsn$ are projectively induced hence of
  algebraic/polynomial character.  Theorem \ref{main} can be
  considered as an analogue for functions of Tian theorem. Indeed
  denote by $\susno $ the set of functions of the form $(\Mom_N \circ
  \phi_N)^* \la $ where $\la$ varies in the dual of $\suv$.  $\susno$
  is a finite dimensional subspace of $\cinfr$ whose elements are real
  algebraic since $\Mom_N$ is a real algebraic mapping. By
  construction the range of $\pzn$ is contained in $\susno$ and the
  range of $P_N$ is contained in 
  \begin{gather*}
    \susn := \R + \susno \subset
  \cinfr
  \end{gather*}
  where $\R$ denotes the constant functions. Hence for any $f \in
  \cinf(M)$, the function $P_Nf$ is real algebraic function and
  projectively induced in a sense.
\end{say}

\begin{say}
  The elements of $\susno $ are related to a particular kind of
  holomorphic functions $L^* \times \overline{L^*}$, the so-called
  Hermitian algebraic functions, see \cite[Def. 2.1
  p. 297]{varolin-geometry-hermitian}.
\end{say}

\begin{say}
  The construction of the operators $P_N^0$ and $P_N$ is of course
  inspired by Berezin-Toeplitz quantization and especially by a paper
  of Bordemann, Meinrenken, Schlichenmaier
  \cite{bordemann-meinrenken-schlichenmaier} (see also
  \cite{schlichenmaier-thesis}). In that paper the authors did not use
  explicitly the Toeplitz operators to construct new functions on
  $M$. Nevertheless what they showed in the proof of Theorem 4.1 is
  equivalent, in our notation, to the fact that $P_N f (z_0) \to
  f(z_0)$ at points $z_0\in M$ where $|f|$ attains its maximum.
\end{say}

\begin{say}
  Let $\tp_N : \cinf(X) \ra \cinf(X)$ be the operator with kernel
  \begin{gather*}
    \tp_N(x,x') = \frac{|\Pi_N(x,x')|^2}{\Pi_N(x,x)}.
  \end{gather*}
  Then $ \tp_N ( \pi^*f) = \pi^*\left( P_N f\right) $ for any $f\in
  \cinf(M)$ and the relation between $\tp_N$ and $\Pi_N$ is identical
  with that between the Poisson-Szeg\"o and the Szeg\"o kernels of a
  domain $D$, see e.g.  \cite[p. 79] {hua-harmonic-analysis} and
  \cite{krantz-construction-Hua}.  Since the Poisson-Szeg\"o operator
  reproduces holomorphic functions in $\hardy(\partial D)$, one might
  wonder if something similar holds for $\tp_N$. This is not the
  case. Indeed the trick used to prove the reproducing property for
  the Poisson-Szeg\"o operator \cite[p. 65]{krantz-function-theory}
  fails in our situation, since for $f\in \hardy_K(X)$ the function
  \begin{gather*}
    u(x') = f(x') \frac{ \overline{\Pi_N(x, x')}} {\Pi_N(x,x)}
  \end{gather*}
  belongs to $\hardy_{N+K}(X)$. So $\tp_N f =f $ if and only if $f\in
  \hardy_0(X)$. This means that $\tp_N$ reproduces only the constant
  functions.
\end{say}

\begin{say}
  One might try to prove (b) in Theorem \ref{main} by another path,
  using the following result.  Consider $T_{f,N}$ as an operator on
  $\cinf(X)$ and denote by $T_{f,N}(x,x')$ its Schwartz kernel.  The
  function $T_{f,N}(x,x)$ descends to a function $\T_N$ on $M$ and
  \begin{gather}
    \label{eq:asto}
    \T_N= \left (\frac{N}{2\pi}\right )^m\cdot f + O(N^{m-1})
  \end{gather}
  in the topology of $\cinf(M)$.  This follows from the main result in
  \cite{charles-Berezin-Toeplitz} and can also be established using
  the scaling asymptotics of the Szeg\"o kernel (established in \cite
  {shiffman-zelditch-Crelle, shiffman-zelditch-number}), see
  \cite[Lemma 7.2.4]{ma-marinescu-libro}. (See also \cite
  {ma-marinescu-Berezin-Toeplitz} for a study of higher terms in the
  asymptotic development.)  If $M_f: L^2(X) \ra L^2(X)$ denotes the
  operator of multiplication by $f$, then $T_{f,N} = \Pi_N M_f \Pi_N$
  so
  \begin{gather*}
    T_{f,N}(x,x'') = \int_X \Pi_N(x,x') f\left (\pi(x')\right )
    \Pi_N(x', x'')  \vx(x') \\
    \T_N(z) = T_{f,N}(x,x) = \int_X | \Pi_N(x,x')|^2 f\left (\pi(x')
    \right) \vx(x') = \\
    = \int_M K_N(z,z') f(z') \vm(z').
  \end{gather*}
  Therefore
  \begin{gather*}
    P_N f (z) = \frac{\T_N(z)} {E_N(z)}.
  \end{gather*}
  It follows from \eqref{eq:zelditch} and \eqref{eq:asto} that for
  every $f\in \cinf(M)$ and any $k$ there is $C'_k$ such that $||P_N f
  - f||_{C^k(X)} \leq C'_k/N$.  The constant $C'_k$ depends on $f \in
  \cinf(M)$ and we would like to show that it can be chosen in the
  form $C'_k=C_k ||f||_{C^{k+1}(M)}$. Unfortunately it is not clear
  how to accomplish this last step, so this method of proof is
  incomplete.
\end{say}

\begin{say}
  If $M =\PP^1, L=\OO_{\PP^1}(2)$ and $\om = 2 \om_{FS}$, i.e. if
  $M=S^2$ with the round metric, it would be interesting to relate
  this approximation procedure to more classical constructions.  The
  Szeg\"o kernel of $(\PP^1, \OO_{\PP^1}(1) , \om_{FS})$ is
  \begin{gather*}
    \Pi_N (x,x') = \frac{N+1}{2\pi} \sx x, x'\xs^N
  \end{gather*}
  where $x, x'\in X = S^3 \subset \C^2$ and $\sx \ , \ \xs$ denotes
  the Hermitian product on $\C^2$ (see
  e.g. \cite[p. 65]{schlichenmaier-thesis}).  So for $(M,L, \om)=\left
    (\PP^1, \OO_{\PP^1}(2), 2\om_{FS}\right )$ 
  \begin{gather*}
    \Pi_N (x,x') = \frac{2N+1}{4\pi} \sx x, x'\xs^N
  \end{gather*}
  where $x, x'\in X = S^3 /\{\pm 1\} = \operatorname{SO}(3)$.  Hence
  $P_N=Q_N$ and the Schwartz kernel is
  \begin{gather*}
    P_N (z,z') = \frac{2N+1}{4\pi} \left |\sx x, x'\xs \right |^{2N} 
    \qquad \pi(x) =z, \pi(x') = z'.
  \end{gather*}
  If we identify $\PP^1$ with the sphere $S^2 \subset \R^3$ in the
  standard way, this becomes
  \begin{gather*}
    P_N(y,y') = \frac{2N+1}{4\pi} \left ( \frac {1 + y\cdot y'} {2}
    \right)^N \qquad y,y' \in S^2
  \end{gather*}
  where $y \cdot y' $ denotes the scalar product in $\R^3$.  Apart
  from this expression, which looks quite nice, one might try to
  express the operators $P_N$ in terms of spherical harmonics on
  $S^2$. In fact this computation has already been carried out by
  Donaldson in \cite[p. 612ff]{donaldson-numerical-results} (recall
  that $P_N=Q_N$ in this case).  Since $L^2(S^2) =
  \bigoplus_{m=0}^\infty \harm_m$, where $\harm_m$ is the space of
  degree $k$ spherical harmonics, and the operators $P_N$ are
  $\SO(3)$-equivariant, $P_N(\harm_m ) \subset \harm_m$ and $P_N$ acts
  on $\harm_m$ as multiplication by a scalar $\chi_{m,2N}$.  The range
  of $P_N$ is $\sum_{m=0}^{2N} \harm_m$. One might expect (and the
  writer did hope) that $P_N$ is just orthogonal projection onto its
  range. This is equivalent to $\chi_{m,2N} = 1$ for $m\leq 2N$, but
  Donaldson's computation shows this to be false. The only thing which
  is evident from Donaldson's formula is that $\chi_{m,2N} \to 1$ for
  $N\to \infty$, which is the equivalent to Theorem 5 in the case at
  hand.
\end{say}

\def\cprime{$'$}

\vskip.3cm

\noindent Universit\`a di Milano Bicocca,\\
\textit{E-mail:} \texttt{alessandro.ghigi@unimib.it}


\begin{thebibliography}{10}

\bibitem{bordemann-meinrenken-schlichenmaier}
M.~Bordemann, E.~Meinrenken, and M.~Schlichenmaier.
\newblock Toeplitz quantization of {K}\"ahler manifolds and {${\rm gl}(N)$},
  {$N\to\infty$} limits.
\newblock {\em Comm. Math. Phys.}, 165(2):281--296, 1994.

\bibitem{bouche-convergence}
T.~Bouche.
\newblock Convergence de la m\'etrique de {F}ubini-{S}tudy d'un fibr\'e
  lin\'eaire positif.
\newblock {\em Ann. Inst. Fourier (Grenoble)}, 40(1):117--130, 1990.

\bibitem{catlin-Tian}
D.~Catlin.
\newblock The {B}ergman kernel and a theorem of {T}ian.
\newblock In {\em Analysis and geometry in several complex variables ({K}atata,
  1997)}, Trends Math., pages 1--23. Birkh\"auser Boston, Boston, MA, 1999.

\bibitem{charles-Berezin-Toeplitz}
L.~Charles.
\newblock Berezin-{T}oeplitz operators, a semi-classical approach.
\newblock {\em Comm. Math. Phys.}, 239(1-2):1--28, 2003.

\bibitem{donaldson-numerical-results}
S.~K. Donaldson.
\newblock Some numerical results in complex differential geometry.
\newblock {\em Pure Appl. Math. Q.}, 5(2, Special Issue: In honor of Friedrich
  Hirzebruch. Part 1):571--618, 2009.

\bibitem{fine-Calabi}
J.~Fine.
\newblock Calabi flow and projective embeddings.
\newblock {\em J. Differential Geom.}, 84(3):489--523, 2010.
\newblock Appendix written by Kefeng Liu \& Xiaonan Ma.

\bibitem{hua-harmonic-analysis}
L.~K. Hua.
\newblock {\em Harmonic analysis of functions of several complex variables in
  the classical domains}.
\newblock Translated from the Russian by Leo Ebner and Adam Kor\'anyi. American
  Mathematical Society, Providence, R.I., 1963.

\bibitem{krantz-function-theory} S.~G. Krantz.  \newblock {\em
    Function theory of several complex variables}.  \newblock
  Wadsworth \& Brooks/Cole Advanced Books \& Software, Pacific Grove,
  CA, second edition, 1992.

\bibitem{krantz-construction-Hua}
S.~G. Krantz.
\newblock On a construction of {L}. {H}ua for positive reproducing kernels.
\newblock {\em Michigan Math. J.}, 59(1):211--230, 2010.

\bibitem{ma-marinescu-libro}
X.~Ma and G.~Marinescu.
\newblock {\em Holomorphic {M}orse inequalities and {B}ergman kernels}, volume
  254 of {\em Progress in Mathematics}.
\newblock Birkh\"auser Verlag, Basel, 2007.

\bibitem{ma-marinescu-Berezin-Toeplitz}
X.~Ma and G.~Marinescu.
\newblock Berezin-{T}oeplitz quantization on {K}\"ahler manifolds.
\newblock {\em \texttt{arXiv:math.DG/1009.4405}}, 2010.
\newblock Preprint.

\bibitem{schlichenmaier-thesis}
M.~Schlichenmaier.
\newblock {\em Zwei Anwendungen algebraisch-geometrischer Methoden in der
  theoretischen Physik}.
\newblock 1996.
\newblock Universit\"at Mannheim.
  \texttt{http://math.uni.lu/}\texttt{schlichenmaier/preprints}\texttt{/method%
en.ps.gz}.

\bibitem{shiffman-zelditch-Crelle}
B.~Shiffman and S.~Zelditch.
\newblock Asymptotics of almost holomorphic sections of ample line bundles on
  symplectic manifolds.
\newblock {\em J. Reine Angew. Math.}, 544:181--222, 2002.

\bibitem{shiffman-zelditch-number}
B.~Shiffman and S.~Zelditch.
\newblock Number variance of random zeros on complex manifolds.
\newblock {\em Geom. Funct. Anal.}, 18(4):1422--1475, 2008.

\bibitem{tian-polarized}
G.~Tian.
\newblock On a set of polarized {K}\"ahler metrics on algebraic manifolds.
\newblock {\em J. Differential Geom.}, 32(1):99--130, 1990.

\bibitem{varolin-geometry-hermitian}
D.~Varolin.
\newblock Geometry of {H}ermitian algebraic functions. {Q}uotients of squared
  norms.
\newblock {\em Amer. J. Math.}, 130(2):291--315, 2008.

\bibitem{zelditch-theorem-Tian}
S.~Zelditch.
\newblock Szeg{\H o} kernels and a theorem of {T}ian.
\newblock {\em Internat. Math. Res. Notices}, (6):317--331, 1998.

\end{thebibliography}
\end{document}